\documentclass[12pt]{amsart}
\usepackage{a4wide}
\usepackage{graphicx,eepic, bm, times}
\usepackage{amsthm,amsmath,amsfonts,amssymb}
\usepackage{multirow}

\newtheorem{lemma}{Lemma}[section]

\newtheorem{prop}[lemma]{Proposition}
\newtheorem{thm}[lemma]{Theorem}
\newtheorem{cor}[lemma]{Corollary}
\theoremstyle{definition}

\newtheorem{example}[lemma]{Example}
\newtheorem{conjecture}[lemma]{Conjecture}

\theoremstyle{remark}
\newtheorem{remark}[lemma]{Remark}

\numberwithin{equation}{section} \numberwithin{table}{section}

\begin{document}

\title[Approximation properties of $\beta$-expansions]{Approximation properties of $\beta$-expansions}
\author{Simon Baker}
\address{
School of Mathematics, The University of Manchester,
Oxford Road, Manchester M13 9PL, United Kingdom. E-mail:
simonbaker412@gmail.com}

\date{\today}

\begin{abstract}
Let $\beta\in(1,2)$ and $x\in [0,\frac{1}{\beta-1}]$. We call a sequence $(\epsilon_{i})_{i=1}^{\infty}\in\{0,1\}^{\mathbb{N}}$ a $\beta$-expansion for $x$ if $x=\sum_{i=1}^{\infty}\epsilon_{i}\beta^{-i}$. We call a finite sequence $(\epsilon_{i})_{i=1}^{n}\in\{0,1\}^{n}$ an $n$-prefix for $x$ if it can be extended to form a $\beta$-expansion of $x$. In this paper we study how good an approximation is provided by the set of $n$-prefixes.

Given $\Psi:\mathbb{N}\to\mathbb{R}_{\geq 0}$, we introduce the following subset of $\mathbb{R}$
$$W_{\beta}(\Psi):=\bigcap_{m=1}^{\infty}\bigcup_{n=m}^{\infty}\bigcup_{(\epsilon_{i})_{i=1}^{n}\in\{0,1\}^{n}}\Big[\sum_{i=1}^{n}\frac{\epsilon_{i}}{\beta^{i}},\sum_{i=1}^{n}\frac{\epsilon_{i}}{\beta^{i}}+\Psi(n)\Big]$$ In other words, $W_{\beta}(\Psi)$ is the set of $x\in\mathbb{R}$ for which there exists infinitely many solutions to the inequalities $$0\leq x-\sum_{i=1}^{n}\frac{\epsilon_{i}}{\beta^{i}}\leq \Psi(n).$$ When $\sum_{n=1}^{\infty}2^{n}\Psi(n)<\infty$ the Borel-Cantelli lemma tells us that the Lebesgue measure of $W_{\beta}(\Psi)$ is zero. When $\sum_{n=1}^{\infty}2^{n}\Psi(n)=\infty,$ determining the Lebesgue measure of $W_{\beta}(\Psi)$ is less straightforward. Our main result is that whenever $\beta$ is a Garsia number and $\sum_{n=1}^{\infty}2^{n}\Psi(n)=\infty$ then $W_{\beta}(\Psi)$ is a set of full measure within $[0,\frac{1}{\beta-1}]$. Our approach makes no assumptions on the monotonicity of $\Psi,$ unlike in classical Diophantine approximation where it is often necessary to assume $\Psi$ is decreasing.

\end{abstract}

\maketitle

\section{Introduction}
Let $\beta\in(1,2)$ and $I_{\beta}:=[0,\frac{1}{\beta-1}]$. Given $x\in I_{\beta}$ we say that a sequence $(\epsilon_{i})_{i=1}^{\infty}\in\{0,1\}^{\mathbb{N}}$ is a \emph{$\beta$-expansion} for $x$ if the following equation holds
\begin{equation}
\label{beta expansion equation}
x=\sum_{i=1}^{\infty}\frac{\epsilon_{i}}{\beta^{i}}.
\end{equation}It is a simple exercise to show that $x$ has a $\beta$-expansion if and only if $x\in I_{\beta}.$ Expansions of this form were pioneered in the papers of Parry \cite{Parry} and R\'enyi \cite{Renyi}. One significant difference between integer base expansions and $\beta$-expansions, is that almost every $x\in I_{\beta}$ has uncountably many $\beta$-expansions, unlike in the integer base case where every number has a unique expansion except for a countable set of exceptions which have precisely two. Whenever we use the phrase ``almost every," we always means with respect to Lebesgue measure. The fact that almost every $x\in I_{\beta}$ has uncountably many $\beta$-expansions is due to Sidorov \cite{Sid}.

We say that a finite sequence $(\epsilon_{i})_{i=1}^{n}\in\{0,1\}^{n}$ is an \emph{$n$-prefix for $x$} if there exists $(\epsilon_{n+i})_{i=1}^{\infty}\in\{0,1\}^{\mathbb{N}}$ such that $$x=\sum_{i=1}^{n}\frac{\epsilon_{i}}{\beta^{i}}+\sum_{i=1}^{\infty}\frac{\epsilon_{n+i}}{\beta^{n+i}}.$$ So an $n$-prefix for $x$ is simply any sequence of length $n$ that can be extended to form a $\beta$-expansion for $x$. It is straightforward to show that a sequence $(\epsilon_{i})_{i=1}^{n}\in\{0,1\}^{n}$ is an $n$-prefix for $x$ if and only if
\begin{equation}
\label{prefix equation}
0\leq x-\sum_{i=1}^{n}\frac{\epsilon_{i}}{\beta^{i}}\leq \frac{1}{\beta^{n}(\beta-1)}.
\end{equation}When $(\epsilon_{i})_{i=1}^{n}\in\{0,1\}^{n}$ is an $n$-prefix for $x,$ we also define the number $\sum_{i=1}^{n}\epsilon_{i}\beta^{-i}$ to be an $n$-prefix for $x$. Whether we are referring to a sequence or a number should be clear from the context. We refer to any number of the form $\sum_{i=1}^{n}\epsilon_{i}\beta^{-i}$ as a \emph{level $n$ sum}.

In this paper we study how well a typical $x\in I_{\beta}$ can be approximated by its prefixes. To this end we introduce the following general setup. Let $\Psi:\mathbb{N}\to\mathbb{R}_{\geq 0}$ and $$W_{\beta}(\Psi) :=\bigcap_{m=1}^{\infty}\bigcup_{n=m}^{\infty}\bigcup_{(\epsilon_{i})_{i=1}^{n}\in\{0,1\}^{n}}\Big[\sum_{i=1}^{n}\frac{\epsilon_{i}}{\beta^{i}},\sum_{i=1}^{n}\frac{\epsilon_{i}}{\beta^{i}}+\Psi(n)\Big].$$
Alternatively, $W_{\beta}(\Psi)$ is the set of $x\in \mathbb{R}$ such that for infinitely many $n\in\mathbb{N}$ there exists a level $n$ sum satisfying the inequalities
\begin{equation}
\label{approx equation}
0\leq x-\sum_{i=1}^{n}\frac{\epsilon_{i}}{\beta^{i}}\leq \Psi(n).
\end{equation} Our goal is to understand how well a typical $x\in I_{\beta}$ is approximated by its prefixes. In (\ref{approx equation}) the approximation to $x$ is given by a level $n$ sum, not necessarily an $n$-prefix for $x$. However, as the following argument shows, if (\ref{approx equation}) is satisfied by a level $n$ sum then it must also be satisfied by an $n$-prefix for $x$. For if $(\epsilon_{i})_{i=1}^{n}$ satisfies (\ref{approx equation}) and $(\epsilon_{i})_{i=1}^{n}$ is not an $n$-prefix for $x$, then $\Psi(n)> (\beta^{n}(\beta-1))^{-1}$ by (\ref{prefix equation}). Every element of $I_{\beta}$ has an $n$-prefix for each $n\in\mathbb{N}$. Let us denote the $n$-prefix for $x$ by $(\epsilon_{i}')_{i=1}^{n}.$ Applying (\ref{prefix equation}) we see that $$0\leq x-\sum_{i=1}^{n}\frac{\epsilon_{i}'}{\beta^{i}}\leq \frac{1}{\beta^{n}(\beta-1)}<\Psi(n).$$ Therefore, if $x\in W_{\beta}(\Psi)$ then there exists infinitely many $n$-prefixes for $x$ satisfying (\ref{approx equation}).

When $\sum_{n=1}^{\infty} 2^{n}\Psi(n)<\infty$ the Borel-Cantelli lemma tells us that $\lambda(W_{\beta}(\Psi))=0.$ Here and throughout $\lambda(\cdot)$ denotes the Lebesgue measure. Motivated by observations and results from metric number theory, we expect that if $\sum_{n=1}^{\infty} 2^{n}\Psi(n)=\infty$ and the level $n$ sums are distributed sufficiently uniformly throughout $I_{\beta}$ then $W_{\beta}(\Psi)$ is a set of full measure within $I_{\beta}$.

With the above in mind we introduce the following definition. We say that $\beta$ is \emph{approximation regular} if for each $\Psi:\mathbb{N}\to\mathbb{R}_{\geq 0}$ satisfying $\sum_{n=1}^{\infty} 2^{n}\Psi(n)=\infty,$ we have $W_{\beta}(\Psi)$ is a set of full measure within $I_{\beta}$. We make the following conjecture.

\begin{conjecture}
\label{conjecture 1}
Almost every $\beta\in(1,2)$ is approximation regular.
\end{conjecture} We cannot hope to extend this almost every statement to an every statement. For example, if we take $\beta$ to be a Pisot number, i.e., a real algebraic integer strictly greater than $1$ whose conjugates all have modulus strictly less than $1$. Then the cardinality of the set of level $n$ sums is of the order $\beta^{n}.$ Taking $\Psi(n)=2^{-n}$ it is clear that $\sum_{n=1}^{\infty}2^{n}\Psi(n)=\infty.$ However a simple covering argument appealing to the Borel-Cantelli lemma implies $\lambda(W_{\beta}(\Psi))=0.$

In this paper we fail to prove Conjecture \ref{conjecture 1}. Instead we show that whenever $\beta$ is a special type of algebraic integer known as a Garsia number then $\beta$ is approximation regular. For our purposes a \emph{Garsia number} is a positive real algebraic integer with norm $\pm 2$, whose conjugates are all of modulus strictly greater than $1.$ Recall that the norm of an algebraic integer $\beta$ is defined to be the product of $\beta$ with all of its conjugates. The reader should be aware that in the literature Garsia numbers are not always defined to be positive, and in some cases are taken to be complex. Garsia numbers were first studied as a separate significant class of algebraic integers in a paper by Garsia \cite{Gar}. For more on Garsia numbers we refer the reader to the paper of Hare and Panju \cite{HP} and the references therein.

Our main result is the following.
\begin{thm}
\label{main theorem}
Let $\beta\in(1,2)$ be a Garsia number. Then $\beta$ is approximation regular.
\end{thm}
\begin{remark}
It is worth commenting on the fact that throughout this paper we have imposed no restrictions on the monotonicity of $\Psi$. In classical Diophantine approximation, when $\Psi:\mathbb{N}\to\mathbb{R}_{\geq 0}$ is decreasing the set $$W(\Psi):=\Big\{x\in\mathbb{R}: \textrm{ there exists infinitely many } (p,q)\in\mathbb{Z}\times\mathbb{N} \textrm{ such that }\Big|x-\frac{p}{q}\Big|\leq \Psi(q)\Big\}$$ is either null or full with respect to Lebesgue measure depending on whether $\sum_{q=1}^{\infty} q\Psi(q)$ converges or diverges. In \cite{DufSch} Duffin and Schaeffer showed that it is not possible to relax the monotonicity assumption on $\Psi$. They constructed a function $\Psi:\mathbb{N}\to\mathbb{R}_{\geq 0}$ such that $\sum_{q=1}^{\infty} q\Psi(q)=\infty$ yet $\lambda(W(\Psi))=0.$
\end{remark}

Suppose $\beta$ is approximation regular and $\Psi:\mathbb{N}\to\mathbb{R}_{\geq 0}$ satisfies $\sum_{n=1}^{\infty}2^{n}\Psi(n)=\infty.$ For a Lebesgue generic $x\in I_{\beta}$ it is natural to ask whether $x$ has a $\beta$-expansion $(\epsilon_{i})_{i=1}^{\infty}\in\{0,1\}^{\mathbb{N}}$ such that the inequalities $$0\leq x-\sum_{i=1}^{n}\frac{\epsilon_{i}}{\beta^{i}}\leq \Psi(n)$$ are satisfied for infinitely many $n\in\mathbb{N}$. This turns out to be the case whenever $\Psi$ satisfies a mild technical condition. We say that $\Psi:\mathbb{N}\to\mathbb{R}_{\geq 0}$ is \textit{decaying regularly} if for each $m\in\mathbb{N}$ there exists $C_{m}\in\mathbb{N}$ such that
\begin{equation}
\label{decay reg 2}
\frac{\Psi(n+m)}{\Psi(n)}\geq \frac{1}{C_{m}}
\end{equation}holds for every $n\in\mathbb{N}.$ We emphasise that the constant $C_{m}$ is allowed to depend on $m$. As an example, when $\Psi(n)=2^{-n}$ then $\Psi$ is decaying regularly. For each $m\in\mathbb{N}$ we can take $C_{m}=2^{m}.$

\begin{thm}
\label{beta theorem}
Let $\beta$ be approximation regular and suppose $\Psi:\mathbb{N}\to\mathbb{R}_{\geq 0}$ is decaying regularly and satisfies $\sum_{n=1}^{\infty}2^{n}\Psi(n)=\infty$. Then for almost every $x\in I_{\beta}$ there exists a $\beta$-expansion for $x$ satisfying the inequalities $$0\leq x-\sum_{i=1}^{n}\frac{\epsilon_{i}}{\beta^{i}}\leq \Psi(n)$$ for infinitely many $n\in\mathbb{N}.$
\end{thm}As an application of Theorem \ref{main theorem} and Theorem \ref{beta theorem} we have the following result.
\begin{cor}
Let $\beta\in(1,2)$ be a Garsia number. Then for almost every $x\in I_{\beta}$ there exists a $\beta$-expansion of $x$ which satisfies the inequalities $$0\leq x-\sum_{i=1}^{n}\frac{\epsilon_{i}}{\beta^{i}}\leq \frac{1}{n2^{n}\log n}$$ for infinitely many $n\in\mathbb{N}.$
\end{cor}
In Section \ref{proof of main theorem} we prove Theorem \ref{main theorem} and in Section \ref{beta theorem section} we prove Theorem \ref{beta theorem}. In Section \ref{last section} we discuss the connection between the set $I_{\beta}\setminus W_{\beta}(\Psi)$ and the set of points with a unique $\beta$-expansion. We end our introduction by giving a summary of related work undertaken by other authors.

In two recent papers by Persson and Reeve \cite{PerRev, PerRevA}, the authors considered a setup similar to that of our own. Let $$K_{\beta}(\Psi):=\bigcap_{m=1}^{\infty}\bigcup_{n=m}^{\infty}\bigcup_{(\epsilon_{i})_{i=1}^{n}\in\{0,1\}^{n}}\Big[\sum_{i=1}^{n}\frac{\epsilon_{i}}{\beta^{i}}-\Psi(n),\sum_{i=1}^{n}\frac{\epsilon_{i}}{\beta^{i}}+\Psi(n)\Big].$$ Notice that $W_{\beta}(\Psi)\subseteq K_{\beta}(\Psi).$ In the definition of $K_{\beta}(\Psi)$ the level $n$ sums form the centres of the significant intervals. Whereas in the definition of $W_{\beta}(\Psi)$ the level $n$ sums are the left endpoints of the significant intervals. The reason we have insisted on the level $n$ sums being the left endpoints is because we are interested in the approximation provided by an $n$-prefix, rather than a general level $n$ sum. It is an obvious consequence of (\ref{prefix equation}) that if $x<\sum_{i=1}^{n}\epsilon_{i}\beta^{-i}$ then $(\epsilon_{i})_{i=1}^{n}\in\{0,1\}^{n}$ cannot be an $n$-prefix for $x$.

Persson and Reeve studied the set $K_{\beta}(\Psi)$ when $\Psi(n)=2^{-\alpha n}$ for some $\alpha\in(1,\infty)$. In this case $\sum_{n=1}^{\infty}2^{n}\Psi(n)$ always converges. Motivated by Falconer \cite{Falconer} they studied the intersection properties of $K_{\beta}(\Psi)$. In \cite{Falconer} Falconer defined $G^{s}$ to be the set of $A\subseteq \mathbb{R},$ which have the property that for any countable collection of similarities $\{f_{j}\}_{j=1}^{\infty},$ we have $$\dim_{H}\Big(\bigcap_{j=1}^{\infty}f_{j}(A)\Big)\geq s.$$ Persson and Reeve generalised the definition of $G^{s}$ to arbitrary intervals $I$ by defining $G^{s}(I):=\{A\subseteq I: A+diam(I)\mathbb{Z}\in G^{s}\}.$ The main results of \cite{PerRev, PerRevA} can be summarised in the following theorem.

\begin{thm}
Let $\alpha\in(1,\infty)$ and $\Psi(n)=2^{-\alpha n}$.
\begin{itemize}
  \item For all $\beta\in(1,2),$ $\dim_{H}(K_{\beta}(\Psi))\leq \frac{1}{\alpha}$.
  \item For almost every $\beta\in(1,2),$ $K_{\beta}(\Psi)\in G^{s}(I_{\beta})$ for $s=\frac{1}{\alpha}.$
  \item For a dense set of $\beta\in(1,2),$ $\dim_{H}(K_{\beta}(\Psi))<\frac{1}{\alpha}.$
  \item For all $\beta\in(1,2)$, $K_{\beta}(\Psi)\in G^{s}(I_{\beta})$ for $s=\frac{\log \beta}{\alpha\log 2}.$
  \item For a countable set of $\beta\in(1,2),$ $\dim_{H}( K_{\beta}(\Psi))=\frac{\log \beta}{\alpha\log 2}.$
\end{itemize}
\end{thm}

The approximation properties of $\beta$-expansions were also studied in a paper by Dajani, Komornik, Loreti, and de Vries \cite{Daj}. Given $x\in I_{\beta}$ and $(\epsilon_{i})_{i=1}^{\infty}$ a $\beta$-expansion for $x.$ We say that $(\epsilon_{i})_{i=1}^{\infty}$ is an \textit{optimal expansion} if for every other $\beta$-expansion for $x$ the following holds for all $n\in\mathbb{N},$ $$x-\sum_{i=1}^{n}\frac{\epsilon_{i}}{\beta^{i}}\leq x-\sum_{i=1}^{n}\frac{\epsilon_{i}'}{\beta^{i}}.$$ In other words, a $\beta$-expansion for $x$ is an optimal expansion if for each $n\in\mathbb{N}$ the $n$-prefix $(\epsilon_{i})_{i=1}^{n}$ always provides the closest approximation to $x.$ Before we state the main result of \cite{Daj} we recall the definition of a multinacci number. A \textit{multinacci number} is the unique root of an equation of the form $x^{n}=x^{n-1}+\cdots +x +1$ lying in $(1,2),$ where $n\geq 2$. The golden ratio is a multinacci number, this is the case when $n=2$. It can be shown that every multinacci number is a Pisot number. The main result of \cite{Daj} is the following.

\begin{thm}
\begin{itemize}
  \item Let $\beta$ be a multinacci number, then every $x\in I_{\beta}$ has an optimal expansion.
  \item If $\beta\in (1,2)$ is not a multinacci number, then the set of $x\in I_{\beta}$ with an optimal expansion is nowhere dense and has zero Lebesgue measure.
\end{itemize}

\end{thm}

\section{Preliminaries}

In this section we state the necessary background information from the theory of Bernoulli convolutions. Let $\beta\in(1,2),$ the \emph{Bernoulli convolution} associated to $\beta$ is defined to be the measure $\mu_{\beta}$ where $$\mu_{\beta}(E)= \mathbb{P}\Big(\Big\{(\epsilon_{i})_{i=1}^{\infty}\in\{0,1\}^{\mathbb{N}}: \sum_{i=1}^{\infty}\frac{\epsilon_{i}}{\beta^{i}}\in E\Big\}\Big),$$ for any Borel set $E\subseteq \mathbb{R}$. Here $\mathbb{P}$ is the $(1/2,1/2)$ probability measure on $\{0,1\}^{\mathbb{N}}.$ It is a long standing problem to determine precisely those $\beta$ for which $\mu_{\beta}$ is absolutely continuous with respect to Lebesgue measure. When $\mu_{\beta}$ is absolutely continuous we denote the density function by $h_{\beta}$. We emphasise that the density function is only defined almost everywhere.

Jessen and Wintner showed that $\mu_{\beta}$ is either absolutely continuous with respect to the Lebesgue measure or purely singular \cite{JesWin}. This was later improved upon by Simon and Mauldin \cite{SimMau}, who showed that $\mu_{\beta}$ is either equivalent to the Lebesgue measure or purely singular \cite{SimMau}. Erd\H{o}s in \cite{Erdos} showed that whenever $\beta$ is a Pisot number then $\mu_{\beta}$ is purely singular. No other examples of $\beta\in(1,2)$ for which $\mu_{\beta}$ is singular are known. In a standout paper, Solomyak proved that for almost every $\beta\in(1,2)$ the Bernoulli convolution is absolutely continuous \cite{Solomyak}. This was later improved upon in a paper of Shmerkin \cite{Shm}, where it was shown that the set of $\beta\in(1,2)$ for which $\mu_{\beta}$ is singular has Hausdorff dimension zero. Loosely speaking, it is believed that whenever the level $n$ sums are distributed sufficiently uniformly throughout $I_{\beta},$ then the associated Bernoulli convolution will be absolutely continuous. Similarly, when the level $n$ sums are distributed sufficiently uniformly throughout $I_{\beta}$ we expect $\beta$ to be approximation regular. As such, the results of Shmerkin and Solomyak lend some weight to the validity of Conjecture \ref{conjecture 1}.

The following theorem due to Garsia \cite{Gar} will be essential in our later work.
\begin{thm}
\label{Garsia theorem}
If $\beta\in(1,2)$ is a Garsia number then $\mu_{\beta}$ is absolutely continuous. Moreover, the density of $\mu_{\beta}$ is bounded above by $$\frac{2}{\prod_{i=1}^{k} (\gamma_{i}-1)}.$$ Here $\gamma_{1},\ldots, \gamma_{k}$ are the conjugates of $\beta$.
\end{thm} Garsia numbers are the largest explicit class of real numbers for which it is known that $\mu_{\beta}$ is always absolutely continuous.

Our proof of Theorem \ref{main theorem} also requires the following results taken from Kempton \cite{Kempton}. These results emphasise the connection between $\beta$-expansions and Bernoulli convolutions. Given $\beta\in(1,2)$ and $x\in I_{\beta},$ we denote the set of $n$-prefixes for $x$ by $\Sigma_{\beta,n}(x)$. In \cite{Kempton} the author studied the growth rate of $|\Sigma_{\beta,n}(x)|.$ In particular they studied the following limits
$$\underline{f}(x):=\liminf_{n\to\infty}\frac{(\beta-1)\beta^{n}}{2^{n}}|\Sigma_{\beta,n}(x)|,$$ and $$\overline{f}(x):=\limsup_{n\to\infty}\frac{(\beta-1)\beta^{n}}{2^{n}}|\Sigma_{\beta,n}(x)|.$$ The main results of this paper are the following two theorems.
\begin{thm}
\label{Tom1}
The Bernoulli convolution $\mu_{\beta}$ is absolutely continuous if and only if $$0<\int_{I_{\beta}}\overline{f}(x)dx<\infty.$$ In this case the density $h_{\beta}$ of $\mu_{\beta}$ satisfies $$h_{\beta}(x)=\frac{\overline{f}(x)}{\int_{I_{\beta}}\overline{f}(y)dy}.$$
\end{thm}
\begin{thm}
\label{Tom2}
Suppose that $$0<\int_{I_{\beta}}\underline{f}(x)dx<\infty.$$ Then $\mu_{\beta}$ is absolutely continuous with density function $$h_{\beta}(x)=\frac{\underline{f}(x)}{\int_{I_{\beta}}\underline{f}(y)dy}.$$ Conversely, if $\mu_{\beta}$ is absolutely continuous with bounded density function $h_{\beta}$ then $\underline{f}$ satisfies $$0<\int_{I_{\beta}}\underline{f}(x)dx<\infty.$$
\end{thm}
When $\beta\in(1,2)$ is a Garsia number, Theorem \ref{Garsia theorem} tells us that $\mu_{\beta}$ is absolutely continuous with bounded density function $h_{\beta}$. Combining Theorem \ref{Tom1} and Theorem \ref{Tom2} the following Proposition is immediate.
\begin{prop}
\label{Tom prop}
Let $\beta\in(1,2)$ be a Garsia number and $x\in I_{\beta}$ be such that $h_{\beta}(x)$ is well defined. Then there exists $K_{1}>1$ and $N(x)\in\mathbb{N}$ sufficiently large such that for all $n\geq N(x)$ $$\frac{h_{\beta}(x)}{K_{1}}\leq \frac{\beta^{n}}{2^{n}}|\Sigma_{\beta,n}(x)|\leq K_{1}h_{\beta}(x).$$ Here $K_{1}$ only depends on $\beta.$
\end{prop}Proposition \ref{Tom prop} will be a vital tool when it comes to proving Theorem \ref{main theorem}.

\section{Proof of Theorem \ref{main theorem}}
\label{proof of main theorem}
Our proof of Theorem \ref{main theorem} is inspired by the work of Beresnevich \cite{BerA,Ber}. However, it is not a simple case of swapping notation where appropriate, a much more delicate argument is required.

We start by proving several technical lemmas. The following lemma is due to Garsia \cite{Gar}.

\begin{lemma}
\label{Garsia's lemma}
Let $\beta\in(1,2)$ be a Garsia number and $(\epsilon_{i})_{i=1}^{n},(\epsilon_{i}')_{i=1}^{n}\in\{0,1\}^{n}.$ If $(\epsilon_{i})_{i=1}^{n}\neq (\epsilon_{i}')_{i=1}^{n}$ then $$\Big|\sum_{i=1}^{n}\frac{\epsilon_{i}}{\beta^{i}}-\sum_{i=1}^{n}\frac{\epsilon_{i}'}{\beta^{i}}\Big|> \frac{K_{2}}{2^{n}}.$$ For some strictly positive constant $K_{2}$ that only depends on $\beta.$
\end{lemma} The proof of Lemma \ref{Garsia's lemma} is well known. However to keep our work as self contained as possible we provide a short proof.
\begin{proof}
Let $(\epsilon_{i})_{i=1}^{n},(\epsilon_{i}')_{i=1}^{n}\in\{0,1\}^{n}$ and assume $(\epsilon_{i})_{i=1}^{n}\neq (\epsilon_{i}')_{i=1}^{n}.$ We introduce the following polynomials $$P(z)=\epsilon_{1}z^{n-1}+\cdots + \epsilon_{n-1}z+\epsilon_{n}$$and  $$P'(z)=\epsilon_{1}'z^{n-1}+\cdots + \epsilon_{n-1}'z+\epsilon_{n}'.$$ Since $\beta$ is an algebraic integer with norm $\pm 2$ it satisfies no polynomials with coefficients in $\{-1,0,1\}.$ Therefore $P(\beta)-P'(\beta)\neq 0.$ Moreover, if $\gamma_{1},\ldots, \gamma_{k}$ denotes the conjugates of $\beta$ then
\begin{equation}
\label{garsia 1}
(P(\beta)-P'(\beta))\prod_{i=1}^{k} (P(\gamma_{i})- P'(\gamma_{i}))\in\mathbb{Z}\setminus\{0\}.
\end{equation} Taking the absolute value of (\ref{garsia 1}) and applying a trivial lower bound, we see that (\ref{garsia 1}) implies the following inequalities
\begin{align*}
1&\leq \Big| (P(\beta)-P'(\beta))\prod_{i=1}^{k} (P(\gamma_{i})- P'(\gamma_{i}))\Big|\\
&\leq \Big| P(\beta)-P'(\beta)\Big|\prod_{i=1}^{k}(1+|\gamma_{i}|+\cdots +|\gamma_{i}^{n-1}|)\\
&< \Big| P(\beta)-P'(\beta)\Big|\prod_{i=1}^{k}\frac{|\gamma_{i}^{n}|}{|\gamma_{i}|-1}\\
&\leq \Big|P(\beta)-P'(\beta)\Big|\frac{2^{n}}{\beta^{n}}\prod_{i=1}^{k} \frac{1}{|\gamma_{i}|-1}\\
&= 2^{n}\Big|\sum_{i=1}^{n}\frac{\epsilon_{i}}{\beta^{i}}-\sum_{i=1}^{n}\frac{\epsilon_{i}'}{\beta^{i}}\Big|\prod_{i=1}^{k} \frac{1}{|\gamma_{i}|-1}.
\end{align*} Which implies the required lower bound. In the above we have used the fact $\beta^{n} \prod_{i=1}^{k}|\gamma_{i}|^{n}=2^{n}.$ This follows from the fact that the norm of $\beta$ is $\pm 2.$
\end{proof}
Recall the Lebesgue differentiation theorem. This theorem states that if $f\in L^{1}(\mathbb{R})$ then for almost every $x\in\mathbb{R}$ the following holds
\begin{equation}
\label{Leb dif}
\lim_{r\to 0}\frac{1}{2r}\int_{B_{r}(x)}f(y)d\lambda(y)=f(x).
\end{equation}Here $B_{r}(x)$ denotes the closed interval centred at $x$ with radius $r.$ Given $f\in
L^{1}(\mathbb{R}),$ we call any $x\in\mathbb{R}$ satisfying (\ref{Leb dif}) a \emph{Lebesgue differentiation point for $f.$} The Lebesgue differentiation theorem tells us that given $f\in L^{1}(\mathbb{R}),$ almost every $x\in\mathbb{R}$ is a Lebesgue differentiation point for $f.$ With this theorem in mind we establish the following lemma.

\begin{lemma}
\label{differentiation lemma}
Let $\beta\in(1,2)$ be a Garsia number, and let $x\in I_{\beta}$ be a Lebesgue differentiation point for $h_{\beta}$ satisfying $h_{\beta}(x)>0.$ Let $r^{*}(x)$ be such that $$ \frac{h_{\beta}(x)}{2}\leq\frac{1}{2r}\int_{B_{r}(x)}h_{\beta}(y)d\lambda(y)$$ for all $r\in(0,r^{*}(x))$.
Then there exists $L\in\mathbb{N}$ and $\kappa\in (1,2)$ such that for all $r\in(0,r^{*}(x))$ the following inequality holds $$\lambda\Big(\Big\{y\in B_{r}(x): h_{\beta}(y)\leq \frac{1}{L}\Big\}\Big)\leq \kappa r.$$ Moreover, $L$ and $\kappa$ only depend upon $\beta$ and $x$.
\end{lemma}
\begin{proof}
Fix $\beta$ and $x$ that satisfy the hypothesis of the lemma. We begin by relabelling the upper bound for the density provided by Theorem \ref{Garsia theorem}. Let $$C:=\frac{2}{\prod_{i=1}^{k} (\gamma_{i}-1)}$$ where $\gamma_{1},\ldots, \gamma_{k}$ are the conjugates of $\beta$. To each $L\in\mathbb{N}$ we associate $$A_{L}:=\Big\{y\in B_{r}(x): h_{\beta}(y)\leq \frac{1}{L}\Big\}.$$  For $r\in(0,r^{*}(x))$ the following inequalities hold from the trivial estimates
\begin{align}
\frac{h_{\beta}(x)}{2}&\leq \frac{1}{2r}\Big(\int_{A_{L}}h_{\beta}(y)d\lambda(y)+\int_{B_{r}(x)\setminus A_{L}}h_{\beta}(y)d\lambda(y)\Big)\nonumber \\
&\leq \frac{1}{2r}\Big(\frac{1}{L}\lambda(A_{L})+ (2r-\lambda(A_{L}))C\Big)\label{First inequality}.
\end{align}Manipulating (\ref{First inequality}) yields
\begin{equation}
\label{Second inequality}
\lambda(A_{L})\Big(C-\frac{1}{L}\Big)\leq r(2C-h_{\beta}(x)).
\end{equation}We may assume that $L\in\mathbb{N}$ is sufficiently large that $C-L^{-1}>0.$ In which case
\begin{equation}
\label{measure AL}
\lambda(A_{L})\leq r\Big(\frac{2C-h_{\beta}(x)}{C-1/L}\Big).
\end{equation}As $L\to\infty$ it is obvious that $$\frac{2C-h_{\beta}(x)}{C-1/L}\to \frac{2C-h_{\beta}(x)}{C}.$$ Since $(2C-h_{\beta}(x))C^{-1}\in(1,2),$ we deduce that there exists $L\in\mathbb{N}$ and $\kappa\in(1,2)$ such that for all $r\in(0,r^{*}(x))$ we have $\lambda(A_{L})\leq \kappa r$. Moreover, both $L$ and $\kappa$ only depend upon $x$ and $\beta.$
\end{proof}
We also make use of the following lemma due to Chung and Erd\H{o}s \cite{ChEr}.
\begin{lemma}
\label{Erdos lemma}
Let $(E_{n})_{n=1}^{\infty}$ be a sequence of measurable sets contained in a bounded interval. If the sum $\sum_{n=1}^{\infty}\lambda(E_{n})=\infty$, then we have $$\lambda(\limsup_{n\to\infty}E_{n})\geq \limsup_{k\to\infty}\frac{(\sum_{n=1}^{k}\lambda(E_{n}))^{2}}{\sum_{n=1}^{k}\sum_{m=1}^{k}\lambda(E_{n}\cap E_{m})}.$$
\end{lemma}
We are now in a position to give our proof of Theorem \ref{main theorem}.
\begin{proof}[Proof of Theorem \ref{main theorem}]
The proof of Theorem \ref{main theorem} depends on an application of the Lebesgue density theorem. The Lebesgue density theorem states that if $E\subseteq \mathbb{R}$ is a measurable set, then for almost every $x\in E$ the following holds $$\lim_{r\to 0}\frac{\lambda(E\cap B_{r}(x))}{2r}=1.$$ As a consequence of the Lebesgue density theorem, to show that $W_{\beta}(\Psi)$ is a set of full measure within $I_{\beta},$ it suffices to show that for almost every $x\in I_{\beta}$ there exists $\delta>0$ such that
\begin{equation}
\label{Lebesgue density equation}
\lambda(W_{\beta}(\Psi)\cap B_{r}(x))\geq \delta r.
\end{equation} For all $r$ sufficiently small. Here $\delta$ is allowed to depend on $x$ but is not allowed to depend on $r$. This will be the strategy we employ to show $W_{\beta}(\Psi)$ is of full measure. It is worth noting that the Lebesgue density theorem is simply the Lebesgue differentiation theorem when $f$ is the indicator function on $E.$

For the rest of the proof we fix $x\in I_{\beta}$. We only need to show that (\ref{Lebesgue density equation}) holds for almost every $x\in I_{\beta}.$ We may therefore assume without loss of generality that: $h_{\beta}(x)$ exists, $h_{\beta}(x)>0$, and $x$ is a Lebesgue differentiation point for $h_{\beta}.$ In which case, both Proposition \ref{Tom prop} and Lemma \ref{differentiation lemma} can be applied. The fact that we can take $h_{\beta}(x)>0$ is a consequence of the aforementioned work of Simon and Mauldin \cite{SimMau}, who showed that if $\mu_{\beta}$ is absolutely continuous with respect to the Lebesgue measure then it is in fact equivalent to the Lebesgue measure.

For ease of exposition we break what remains of our proof into three parts.
\\

\noindent $\textbf{(1)}$ \textbf{Replacing $\Psi$ with $\tilde{\Psi}$.}
\\ \\
Let $K_{2}$ be as in Lemma \ref{Garsia's lemma}. So for $(\epsilon_{i})_{i=1}^{n}\neq (\epsilon_{i}')_{i=1}^{n}$ then
\begin{equation}
\label{separation equation}
\Big|\sum_{i=1}^{n}\frac{\epsilon_{i}}{\beta^{i}}-\sum_{i=1}^{n}\frac{\epsilon_{i}'}{\beta^{i}}\Big|> \frac{K_{2}}{2^{n}}.
\end{equation}
Let $\tilde{\Psi}(n)=\min\{\Psi(n),K_{2}2^{-n}\}$ then $\sum_{n=1}^{\infty} 2^{n}\tilde{\Psi}(n)=\infty.$ To see why $\sum_{n=1}^{\infty} 2^{n}\tilde{\Psi}(n)=\infty$ we remark that if $\sum_{n=1}^{\infty} 2^{n}\tilde{\Psi}(n)<\infty$ then there must exist infinitely many $n\in\mathbb{N}$ for which $\tilde{\Psi}(n)=K_{2}2^{-n}.$ This is a consequence of $\sum_{n=1}^{\infty}2^{n}\Psi(n)$ diverging. However, this implies that for infinitely many $n\in\mathbb{N}$ the term $2^{n}\tilde{\Psi}(n)$ equals $K_{2},$ and as $K_{2}>0$ the sum must diverge.

Clearly $W_{\beta}(\tilde{\Psi})\subseteq W_{\beta}(\Psi).$ Therefore, to show that (\ref{Lebesgue density equation}) holds and $W_{\beta}(\Psi)$ is a set of full measure within $I_{\beta}$, it is sufficient to show that the following analogue of (\ref{Lebesgue density equation}) holds for some $\delta>0$ and for all $r$ sufficiently small
\begin{equation}
\label{Lebesgue density equation 2}
\lambda(W_{\beta}(\tilde{\Psi})\cap B_{r}(x))\geq \delta r.
\end{equation}

The important feature of our new function $\tilde{\Psi}$ is that (\ref{separation equation}) implies that for $(\epsilon_{i})_{i=1}^{n}\neq (\epsilon_{i}')_{i=1}^{n}$ we have
\begin{equation}
\label{separation equation 2} \Big[\sum_{i=1}^{n}\frac{\epsilon_{i}}{\beta^{i}},\sum_{i=1}^{n}\frac{\epsilon_{i}}{\beta^{i}}+\tilde{\Psi}(n)\Big]\bigcap\Big[\sum_{i=1}^{n}\frac{\epsilon_{i}'}{\beta^{i}},\sum_{i=1}^{n}\frac{\epsilon_{i}'}{\beta^{i}}+\tilde{\Psi}(n)\Big]=\emptyset.
\end{equation}This observation will prove useful later on in our proof.
\\

\noindent $\textbf{(2)}$ \textbf{Construction of the $E_{n}$.}
\\ \\
Let $r\in(0,r^{*}(x))$ and $L\in\mathbb{N}$ be as in Lemma \ref{differentiation lemma}. Let $$B_{L}:=\Big\{y\in B_{r}(x): h_{\beta}(y)\geq \frac{1}{L}\Big\}.$$ Lemma \ref{differentiation lemma} tells us that $\lambda(B_{L})\geq \omega r$ where $\omega:=2-\kappa>0.$ Importantly $\omega$ only depends upon $\beta$ and $x$. 

Proposition \ref{Tom prop} tells us that for almost every $y\in I_{\beta}$ there exists $N(y)\in\mathbb{N}$ sufficiently large that
\begin{equation}
\label{growth bounds 1}
\frac{h_{\beta}(y)}{K_{1}}\leq \frac{\beta^{n}}{2^{n}}|\Sigma_{\beta,n}(y)|\leq h_{\beta}(y)K_{1}.
\end{equation} for all $n\geq N(y).$ Using the upper bound for the density provided by Theorem \ref{Garsia theorem}, we see that for almost every $y\in B_{L}$ there exists $N(y)\in\mathbb{N}$ such that
\begin{equation}
\label{growth bounds 2}
\frac{1}{LK_{1}}\leq \frac{\beta^{n}}{2^{n}}|\Sigma_{\beta,n}(y)|\leq \frac{2K_{1}}{\prod_{i=1}^{k}(\gamma_{i}-1)}.
\end{equation} for all $n\geq N(y).$ Now let us take $N^{*}\in\mathbb{N}$ to be sufficiently large that
\begin{equation}
\label{inequality 1}
\lambda\Big(\Big\{y\in B_{L}: \frac{1}{LK_{1}}\leq \frac{\beta^{n}}{2^{n}}|\Sigma_{\beta,n}(y)|\leq \frac{2K_{1}}{\prod_{i=1}^{k}(\gamma_{i}-1)} \textrm{ for all } n\geq N^{*}\Big\}\Big)\geq \frac{\omega r}{2}.
\end{equation}Throughout our proof $N^{*}$ is allowed to depend on $r$. Let $$C:=\Big\{y\in B_{L}: \frac{1}{LK_{1}}\leq \frac{\beta^{n}}{2^{n}}|\Sigma_{\beta,n}(y)|\leq \frac{2K_{1}}{\prod_{i=1}^{k}(\gamma_{i}-1)} \textrm{ for all } n\geq N^{*}\Big\}.$$ Upon relabelling, any $y\in C$ satisfies
\begin{equation}
\label{growth bounds 3}
\frac{1}{K_{3}}\leq \frac{\beta^{n}}{2^{n}}|\Sigma_{\beta,n}(y)|\leq K_{3}
\end{equation} for all $n\geq N^{*}$. Where $K_{3}$ is some positive constant depending only upon $\beta$ and $x$. Importantly $K_{3}$ does not depend on $r.$

We now focus our attention on the interval $B_{r}(x).$ Fix $n\geq N^{*}$ where $N^{*}$ is as above. We now fill $B_{r}(x)$ with closed intervals satisfying certain desirable properties. We may pick a set of closed intervals satisfying the following:
\begin{itemize}
\item Each interval is of width $(\beta^{n}(\beta-1))^{-1}.$
  \item Each of these intervals are strictly contained in $B_{r}(x).$
  \item If they intersect it is only at a shared endpoint.
  \item They cover all of $B_{r}(x)$ except for a set of measure at most $\omega r/4.$
\end{itemize}
To assert that a set of intervals satisfying this covering property exist, it is necessary to assume that $N^{*}$ is sufficiently large. This is permissible as $N^{*}$ is allowed to depend on $r$. Let $\{I_{j}^{n}\}$ denote a set of intervals satisfying the above properties. It is a consequence of (\ref{inequality 1}) and the above properties that
\begin{equation}
\label{measure inequality}
\lambda\Big(\bigcup_{j} I_{j}^{n}\cap C\Big)\geq \frac{\omega r}{4}.
\end{equation}Without loss of generality, we may assume that the enumeration of the set $\{I_{j}^{n}\}$ is such that $I_{1}^{n}$ is the leftmost interval, then $I_{2}^{n}$ sits immediately to the right of $I_{1}^{n}$, then $I_{3}^{n}$ sits immediately to the right of $I_{2}^{n}$, and so on. This implies that for any two distinct intervals in $\{I_{j}^{n}\}$ whose subscript have the same parity, there is at least one interval of size $(\beta^{n}(\beta-1))^{-1}$ sitting between them. We partition $\{I_{j}^{n}\}$ into two subsets, those with an odd subscript $\{I_{j, odd}^{n}\}$ and those with an even subscript $\{I_{j,even}^{n}\}.$ It is a consequence of (\ref{measure inequality}) that $$\lambda\Big(\bigcup_{j}I_{j,odd}^{n}\cap C\Big)\geq \frac{\omega r}{8}\textrm{ or }\lambda\Big(\bigcup_{j}I_{j,even}^{n}\cap C\Big)\geq \frac{\omega r}{8}.$$ Without loss of generality we assume that $\lambda(\bigcup I_{j,odd}^{n}\cap C)\geq \frac{\omega r}{8}$. Let $$J:=\{I_{j,odd}^{n}: int(I_{j,odd}^{n})\cap C\neq \emptyset\}.$$
Each $I_{j,odd}^{n}$ is of width $(\beta^{n}(\beta-1))^{-1},$ therefore $$|J|\geq \Big[\frac{\beta^{n}(\beta-1)\omega r}{8}\Big].$$ We pick a subset of $J$ with cardinality precisely $[\frac{\beta^{n}(\beta-1)\omega r}{8}].$ Abusing notation we also denote this set by $J$.

For each $I_{j,odd}^{n}\in J$ we choose a point $\alpha_{j}^{n}\in int(I_{j,odd}^{n})\cap C$. Since $|J|=[\frac{\beta^{n}(\beta-1)\omega r}{8}]$ we have
\begin{equation}
\label{card 1}
|\{\alpha_{j}^{n}\}|=\Big[\frac{\beta^{n}(\beta-1)\omega r}{8}\Big].
\end{equation}For each $\alpha_{j}^{n},$ let $\{\nu_{s,j}^{n}\}$ denote the set of $n$-prefixes $\Sigma_{\beta,n}(\alpha_{j}^{n}).$ We are now in a position to define the set $E_{n}$. Let
\begin{equation}
\label{E_{n}}
E_{n}:=\bigcup_{\alpha_{j}^{n}}\bigcup_{\nu_{s,j}^{n}\in\Sigma_{\beta,n}(\alpha_{j}^{n})}[\nu_{s,j}^{n},\nu_{s,j}^{n}+\tilde{\Psi}(n)].
\end{equation}

For distinct $\alpha_{j}^{n},\alpha_{j'}^{n}$ we have $|\alpha_{j}^{n}-\alpha_{j'}^{n}|>(\beta^{n}(\beta-1))^{-1}.$ This is because $\alpha_{j}^{n}$ and $\alpha_{j'}^{n}$ are in the interior of distinct $I_{j}^{n}$ and $I_{j'}^{n},$ where $j$ and $j'$ have the same parity. Recall that it is as a consequence of our construction that for any two intervals of the same parity there exists an interval of width $(\beta^{n}(\beta-1))^{-1}$ sitting between them. By (\ref{prefix equation}) each element of $\Sigma_{\beta,n}(\alpha_{j}^{n})$ is contained in $[\alpha_{j}^{n}-\frac{1}{\beta^{n}(\beta-1)},\alpha_{j}],$ and similarly each element of $\Sigma_{\beta,n}(\alpha_{j'}^{n})$ is contained in $[\alpha_{j'}^{n}-\frac{1}{\beta^{n}(\beta-1)},\alpha_{j'}^{n}]$. Therefore $\Sigma_{\beta,n}(\alpha_{j}^{n})\cap \Sigma_{\beta,n}(\alpha_{j'}^{n})=\emptyset,$ and by (\ref{separation equation 2}) we may conclude that any two distinct intervals $[\nu_{s,j}^{n},\nu_{s,j}^{n}+\tilde{\Psi}(n)]$ and $[\nu_{s',j'}^{n},\nu_{s',j'}^{n}+\tilde{\Psi}(n)]$ appearing in (\ref{E_{n}}) are disjoint.
Making use of this fact, along with (\ref{growth bounds 3}) and (\ref{card 1}) we observe the following inequalities
\begin{equation}
\label{E_{n} measure}
\Big[\frac{\beta^{n}(\beta-1)\omega r}{8}\Big]\frac{2^{n}}{\beta^{n}K_{3}}\tilde{\Psi}(n)\leq \lambda(E_{n})\leq \Big[\frac{\beta^{n}(\beta-1)\omega r}{8}\Big]\frac{2^{n}K_{3}}{\beta^{n}}\tilde{\Psi}(n).
\end{equation}
It is clear that (\ref{E_{n} measure}) implies
\begin{equation}
\label{E_{n} measure 2}
\frac{2^{n}r}{K_{4}}\tilde{\Psi}(n)\leq \lambda(E_{n})\leq 2^{n}rK_{4}\tilde{\Psi}(n),
\end{equation}for some positive constant $K_{4}$ that only depends upon $\beta$ and $x$.

Clearly $\limsup_{n\to\infty} E_{n}\subset W_{\beta}(\tilde{\Psi})\cap B_{r}(x).$ Therefore to show that there exists $\delta>0$ for which (\ref{Lebesgue density equation 2}) holds, it suffices to show that there exists $\delta>0$ such that
\begin{equation}
\label{Lebesgue density 2}
\lambda(\limsup_{n\to\infty}E_{n})\geq \delta r.
\end{equation}
Equation (\ref{E_{n} measure 2}) and our divergence assumption implies $\sum_{n=N^{*}}^{\infty}\lambda(E_{n})=\infty$. Therefore we can apply Lemma \ref{Erdos lemma}. In the next part of our proof we obtain a lower bound for  $\lambda(\limsup_{n\to\infty} E_{n})$ using Lemma \ref{Erdos lemma}. As we will see this lower bound yields a $\delta$ so that we satisfy (\ref{Lebesgue density 2}).
\\

\noindent $\textbf{(3)}$ \textbf{Applying Lemma \ref{Erdos lemma} to $E_{n}$.}
\\ \\
To begin with, let $M_{0}\in\mathbb{N}$ be sufficiently large that
\begin{equation}
\label{>1 equation}
\sum_{n=N^{*}}^{M_{0}}2^{n}\tilde{\Psi}(n)>1.
\end{equation}Let $m,n\geq N^{*}.$ For any $\nu_{s,j}^{m},$ the number of $\nu_{s',j'}^{n}$ whose corresponding interval $[\nu_{s',j'}^{n},\nu_{s',j'}^{n}+\tilde{\Psi}(n)]$ may intersect $[\nu_{s,j}^{m},\nu_{s,j}^{m}+\tilde{\Psi}(m)]$ is at most $$2+\frac{\tilde{\Psi}(m)}{K_{2}2^{-n}}=2+\frac{2^{n}\tilde{\Psi}(m)}{K_{2}},$$ by Lemma \ref{Garsia's lemma}. Therefore
\begin{equation}
\label{single intersection}
\lambda\Big(E_{n}\cap [\nu_{s,j}^{m},\nu_{s,j}^{m}+\tilde{\Psi}(m)]\Big)\leq \tilde{\Psi}(n)\Big(2+\frac{2^{n}\tilde{\Psi}(m)}{K_{2}}\Big).
\end{equation} Applying (\ref{growth bounds 3}) and (\ref{card 1}) it is clear that $$\Big|\bigcup_{\alpha_{j}^{m}}\Sigma_{\beta,m}(\alpha_{j}^{m})\Big|\leq \Big[\frac{\beta^{m}(\beta-1)\omega r}{8}\Big] \frac{2^{m}}{\beta^{m}}K_{3}.$$
Therefore
\begin{equation}
\label{card 2}
\Big|\bigcup_{\alpha_{j}^{m}}\Sigma_{\beta,m}(\alpha_{j}^{m})\Big|\leq 2^{m}rK_{5}.
\end{equation} Where $K_{5}$ is some positive constant depending only on $\beta$ and $x$. Combining (\ref{single intersection}) with (\ref{card 2}) we obtain the following bound
\begin{equation}
\label{big intersection}
\lambda(E_{n}\cap E_{m})\leq 2^{m}rK_{5}\Big(\tilde{\Psi}(n)\Big(2+\frac{2^{n}\tilde{\Psi}(m)}{K_{2}}\Big)\Big)\leq 2rK_{5}\Big(2^{m}\tilde{\Psi}(n)+\frac{2^{n+m}\tilde{\Psi}(n)\tilde{\Psi}(m)}{K_{2}}\Big).
\end{equation}
We now give an upper bound for the double summation appearing in the denominator in Lemma \ref{Erdos lemma}. First of all we split up the terms in this summation
\begin{equation}
\label{double sum}
\sum_{n=N^{*}}^{M_{0}}\sum_{m=N^{*}}^{M_{0}}\lambda(E_{n}\cap E_{m})=\sum_{n=N^{*}}^{M_{0}}\lambda(E_{n})+2\sum_{n=N^{*}+1}^{M_{0}}\sum_{m=N^{*}}^{n-1}\lambda(E_{n}\cap E_{m}).
\end{equation} By (\ref{E_{n} measure 2}) and (\ref{>1 equation}) we obtain
\begin{equation}
\label{first bound}
\sum_{n=N^{*}}^{M_{0}}\lambda(E_{n})\leq rK_{4} \sum_{n=N^{*}}^{M_{0}}2^{n}\tilde{\Psi}(n)\leq rK_{4} \Big(\sum_{n=N^{*}}^{M_{0}}2^{n}\tilde{\Psi}(n)\Big)^{2}
\end{equation}
As a consequence of (\ref{big intersection}) we obtain
\begin{equation}
\label{second bound}
\sum_{n=N^{*}+1}^{M_{0}}\sum_{m=N^{*}}^{n-1}\lambda(E_{n}\cap E_{m})\leq 2rK_{5}\sum_{n=N^{*}+1}^{M_{0}}\sum_{m=N^{*}}^{n-1}\Big(2^{m}\tilde{\Psi}(n)+\frac{2^{n+m}\tilde{\Psi}(n)\tilde{\Psi}(m)}{K_{2}}\Big).
\end{equation}We now split the summation in (\ref{second bound}) into two summations. For the first summation we have the following bound
\begin{equation}
\label{third bound}
\sum_{n=N^{*}+1}^{M_{0}}\sum_{m=N^{*}}^{n-1}2^{m}\tilde{\Psi}(n)\leq \sum_{n=N^{*}+1}^{M_{0}}2^{n}\tilde{\Psi}(n)\leq \Big(\sum_{n=N^{*}}^{M_{0}}2^{n}\tilde{\Psi}(n)\Big)^{2}.
\end{equation} For the second summation in (\ref{second bound}) we observe
\begin{equation}
\label{forth bound}
\sum_{n=N^{*}+1}^{M_{0}}\sum_{m=N^{*}}^{n-1}2^{n+m}\tilde{\Psi}(n)\tilde{\Psi}(m)\leq \Big(\sum_{n=N^{*}}^{M_{0}}2^{n}\tilde{\Psi}(n)\Big)^{2}.
\end{equation}Combining (\ref{E_{n} measure 2}), (\ref{double sum}), (\ref{first bound}), (\ref{second bound}), (\ref{third bound}) and (\ref{forth bound}) we obtain
\begin{equation}
\label{conclusion equation}
\frac{\Big(\sum_{n=N^{*}}^{M_{0}}\lambda(E_{n})\Big)^{2}}{\sum_{n=N^{*}}^{M_{0}}\sum_{m=N^{*}}^{M_{0}}\lambda(E_{n}\cap E_{m})}\geq \frac{r^{2}K_{4}^{-2}\Big(\sum_{n=N^{*}}^{M_{0}}2^{n}\tilde{\Psi}(n)\Big)^{2}}{r(K_{4}+ 4K_{5}+4K_{2}^{-1}K_{5})\Big(\sum_{n=N^{*}}^{M_{0}}2^{n}\tilde{\Psi}(n)\Big)^{2}}.
\end{equation}Letting $$\delta:=\frac{K_{4}^{-2}}{K_{4}+ 4K_{5}+4K_{2}^{-1}K_{5}}$$ it is clear that $\delta$ only depends on $\beta$ and $x$. Combining Lemma \ref{Erdos lemma} and (\ref{conclusion equation}) we obtain $$\lambda(\limsup_{n\to\infty}E_{n})\geq \delta r.$$ Therefore (\ref{Lebesgue density 2}) holds and we may conclude that $W_{\beta}(\Psi)$ is a set of full measure within $I_{\beta}$.
\end{proof}
\section{Proof of Theorem \ref{beta theorem}}
\label{beta theorem section}
In this section we prove Theorem \ref{beta theorem}. Our proof is straightforward and relies on basic properties of the Lebesgue measure. For ease of exposition we briefly recall the definition of decaying regularly. We say that $\Psi$ is decaying regularly if for each $m\in\mathbb{N}$ there exists $C_{m}\in\mathbb{N}$ such that
\begin{equation}
\label{decay reg}
\frac{\Psi(n+m)}{\Psi(n)}\geq \frac{1}{C_{m}}
\end{equation} for every $n\in\mathbb{N}.$

Suppose $\Psi:\mathbb{N}\to\mathbb{R}_{\geq 0}$ satisfies $\sum_{n=1}^{\infty}2^{n}\Psi(n)=\infty.$ Given $k\in\mathbb{N}$ let $\Psi_{k}:\mathbb{N}\to\mathbb{R}_{\geq 0}$ be defined via the equation $\Psi_{k}(n):=\Psi(n)k^{-1}.$ For each $k\in\mathbb{N}$ the summation $\sum_{n=1}^{\infty}2^{n}\Psi_{k}(n)$ also diverges. If $\beta$ is approximation regular then $W_{\beta}(\Psi_{k})$ is a set of full measure within $I_{\beta}$ for each $k\in\mathbb{N}$. Therefore $$\Omega_{\beta}(\Psi):=\bigcap_{k=1}^{\infty}W_{\beta}(\Psi_{k})$$ is also of full measure. Let $$\Gamma_{\beta}(\Psi):=I_{\beta}\setminus \Omega_{\beta}(\Psi),$$ so if $\beta$ is approximation regular then $\lambda(\Gamma_{\beta}(\Psi))=0.$ We introduce the functions $T_{0}(x)=\beta x$ and $T_{1}(x)=\beta x-1.$ We will denote a typical element of $\{T_{0},T_{1}\}^{n}$ by $a=(a_{1},\ldots,a_{n}).$ Moreover, we let $a(x)$ denote $(a_{n}\circ \cdots \circ a_{1})(x).$ By $\{T_{0},T_{1}\}^{0}$ we denote the set consisting of the identity function. Let $$\Delta_{\beta}(\Psi):=\bigcup_{n=0}^{\infty}\bigcup_{a\in\{T_{0},T_{1}\}^{n}}a^{-1}(\Gamma_{\beta}(\Psi)).$$ Since $T_{0}^{-1}$ and $T_{1}^{-1}$ are both similitudes it follows that $\lambda(\Delta_{\beta}(\Psi))=0$ whenever $\beta$ is approximation regular. We are now ready to prove Theorem \ref{beta theorem}.
\begin{proof}[Proof of Theorem \ref{beta theorem}]
Assume $\beta$ is approximation regular, $\Psi:\mathbb{N}\to\mathbb{R}_{\geq 0}$ is decaying regularly and $\sum_{n=1}^{\infty}2^{n}\Psi(n)=\infty$. Let $x\in I_{\beta}\setminus \Delta_{\beta}(\Psi).$ By the above $I_{\beta}\setminus \Delta_{\beta}(\Psi)$ is a set of full Lebesgue measure within $I_{\beta}$. We now show that $x$ has a $\beta$-expansion $(\epsilon_{i})_{i=1}^{\infty}$ which satisfies $$0\leq x-\sum_{i=1}^{n}\frac{\epsilon_{i}}{\beta^{i}}\leq \Psi(n)$$ for infinitely many $n\in \mathbb{N}.$ Since $x\in I_{\beta}\setminus\Delta_{\beta}(\Psi)$ it is clear that $x\in W_{\beta}(\Psi)$. Therefore there exists infinitely many solutions to the inequalities $$0\leq x-\sum_{i=1}^{n}\frac{\epsilon_{i}}{\beta^{i}}\leq \Psi(n).$$ Let $(\epsilon_{i}^{1})_{i=1}^{n_{1}}$ be the first sequence whose level $n_{1}$ sum satisfies these inequalities. Without loss of generality we may assume $(\epsilon_{i}^{1})_{i=1}^{n_{1}}$ is an $n_{1}$-prefix for $x$. In which case, multiplying through by $\beta^{n_{1}}$ in (\ref{prefix equation}) gives us $$(T_{\epsilon_{n_{1}}^{1}}\circ\cdots \circ T_{\epsilon_{1}^{1}})(x)=\beta^{n_{1}}x-\epsilon_{1}^{1}\beta^{n_{1}-1}-\cdots -\epsilon_{n_{1}-1}^{1}\beta-\epsilon_{n_{1}}^{1} \in I_{\beta}.$$ Let $C^{1}\in\mathbb{N}$ be sufficiently large that
\begin{equation}
\label{C^{1} equation}
\frac{\Psi_{C^{1}}(n)}{\beta^{n_{1}}}\leq \Psi(n+n_{1}),
\end{equation}for all $n\in\mathbb{N}$. Such a $C^{1}$ exists since $\Psi$ is decaying regularly. Since $x\in I_{\beta}\setminus \Delta_{\beta}(\Psi)$ we have $(T_{\epsilon_{n_{1}}^{1}}\circ\cdots \circ T_{\epsilon_{1}^{1}})(x)\in W_{\beta}(\Psi_{C^{1}})$. Therefore there exists $(\epsilon_{1}^{2},\ldots,\epsilon_{n_{2}}^{2})$ such that
\begin{equation}
\label{final equation 1}
(T_{\epsilon_{n_{1}}^{1}}\circ\cdots \circ T_{\epsilon_{1}^{1}})(x)-\sum_{i=1}^{n_{2}}\frac{\epsilon_{i}^{2}}{\beta^{i}}\leq \Psi_{C^{1}}(n_{2}).
\end{equation}Dividing through by $\beta^{n_{1}}$ in (\ref{final equation 1}) and applying (\ref{C^{1} equation}) yields $$x-\sum_{i=1}^{n_{1}}\frac{\epsilon_{i}^{1}}{\beta^{i}}-\frac{1}{\beta^{n_{1}}}\sum_{i=1}^{n_{2}}\frac{\epsilon_{i}^{2}}{\beta^{i}}\leq \frac{\Psi_{C^{1}}(n_{2})}{\beta^{n_{1}}}\leq \Psi(n_{1}+n_{2}) .$$ Without loss of generality we may assume that $(\epsilon_{1}^{1},\ldots, \epsilon_{n_{1}}^{1},\epsilon_{1}^{2},\ldots,\epsilon_{n_{2}}^{2})$ is an $n_{1}+n_{2}$ prefix for $x$.

Since $x\in I_{\beta}\setminus \Delta_{\beta}(\Psi)$ we have $(T_{\epsilon_{n_{2}}^{2}}\circ\cdots \circ T_{\epsilon_{1}^{2}}\circ T_{\epsilon_{n_{1}^{1}}}\circ\cdots \circ T_{\epsilon_{1}^{1}})(x)\in W_{\beta}(\Psi_{k})$ for each $k\in\mathbb{N}$. We choose $C^{2}\in\mathbb{N}$ sufficiently large that $$\frac{\Psi_{C^{2}}(n)}{\beta^{n_{1}+n_{2}}}\leq \Psi(n+n_{1}+n_{2}),$$ for all $n\in\mathbb{N}.$ We then repeat the above argument with $C^{1}$ replaced by $C^{2},$ and $(T_{\epsilon_{n_{1}}^{1}}\circ\cdots \circ T_{\epsilon_{1}^{1}})$ replaced by $(T_{\epsilon_{n_{2}}^{2}}\circ\cdots \circ T_{\epsilon_{1}^{2}}\circ T_{\epsilon_{n_{1}}^{1}}\circ\cdots \circ T_{\epsilon_{1}^{1}})$ to obtain a sequence $(\epsilon_{1}^{3},\ldots,\epsilon_{n_{3}}^{3})$ such that $$x-\sum_{i=1}^{n_{1}}\frac{\epsilon_{i}}{\beta^{i}}-\frac{1}{\beta^{n_{1}}}\sum_{i=1}^{n_{2}}\frac{\epsilon_{i}^{2}}{\beta^{i}}-\frac{1}{\beta^{n_{1}+n_{2}}}\sum_{i=1}^{n_{3}}\frac{\epsilon_{i}^{3}}{\beta^{i}}\leq \Psi(n_{1}+n_{2}+n_{3}).$$ Again we may assume that $(\epsilon_{1}^{1},\ldots, \epsilon_{n_{1}}^{1},\epsilon_{1}^{2},\ldots,\epsilon_{n_{2}}^{2}, \epsilon_{1}^{3},\ldots,\epsilon_{n_{3}}^{3})$ is an $n_{1}+n_{2}+n_{3}$ prefix for $x$.

Repeatedly applying the above procedure we obtain an infinite sequence $(\epsilon_{i})_{i=1}^{\infty}$ which forms a $\beta$-expansion for $x$ and satisfies $$0\leq x-\sum_{i=1}^{n}\frac{\epsilon_{i}}{\beta^{i}}\leq \Psi(n)$$ for infinitely many $n\in \mathbb{N}.$
\end{proof}

\section{Final comments}
\label{last section}
In this final section we make a few comments on the connection between the set of points with a unique $\beta$-expansion and $I_{\beta}\setminus W_{\beta}(\Psi)$. Let $$U_{\beta}:=\Big\{x\in\Big(0,\frac{1}{\beta-1}\Big): x \textrm{ has a unique } \beta\textrm{-expansion}\Big\}.$$ $U_{\beta}$ is a well studied object. It is a consequence of the work of Dar\'{o}czy and Katai \cite{DaKa}, and Erd\H{o}s, Jo\'{o} and Komornik \cite{Erdos2}, that $U_{\beta}$ is nonempty if and only if $\beta\in(\frac{1+\sqrt{5}}{2},2).$ Let $\beta_{c}\approx 1.78723$ be the Komornik-Loreti constant introduced in \cite{KomLor}. Glendinning and Sidorov showed in \cite{GlenSid} that: $U_{\beta}$ is countable if $\beta\in(\frac{1+\sqrt{5}}{2},\beta_{c})$, $U_{\beta_{c}}$ is uncountable with zero Hausdorff dimension, and $U_{\beta}$ has strictly positive Hausdorff dimension if $\beta\in(\beta_{c},2)$. Moreover, $\dim_{H}(U_{\beta})\to 1$ as $\beta\to 2.$

The significance of the set $U_{\beta}$ is that if $x\in U_{\beta}$ then $$\frac{\kappa}{\beta^{n}(\beta-1)}\leq x-\sum_{i=1}^{n}\frac{\epsilon_{i}}{\beta^{i}}\leq \frac{1}{\beta^{n}(\beta-1)}$$for all $n\in\mathbb{N}$. Where $(\epsilon_{i})_{i=1}^{\infty}$ is the unique $\beta$-expansion for $x,$ and $\kappa$ is some strictly positive constant that only depends on $x$. This implies that for any $\Psi(n)=O(\gamma^{-n})$ where $\gamma>\beta$ there are finitely many solutions to the set of inequalities $$0\leq x-\sum_{i=1}^{n}\frac{\epsilon_{i}}{\beta^{i}}\leq \Psi(n).$$  Therefore if $\Psi$ decays sufficiently quickly and $\beta\in (\frac{1+\sqrt{5}}{2},2)$ then $I_{\beta}\setminus W_{\beta}(\Psi)$ is always infinite. We finish with an example that emphasises the above.
\begin{example}
Take $\beta\approx 1.76929,$ the appropriate root of $x^{3}-2x-2=0.$ Then $\beta$ is a Garsia number and by Theorem \ref{main theorem} is approximation regular. In which case if we take $\Psi(n)=2^{-n}$ we have $W_{\beta}(\Psi)$ is of full measure. Yet by the above $I_{\beta}\setminus W_{\beta}(\Psi)$ contains an infinite set.
\end{example}

\end{document}